\documentclass[10pt,mathserif]{article}

\usepackage[english]{babel}
\usepackage{a4wide}
\usepackage{xcolor}
\usepackage[T1]{fontenc}
\usepackage{graphicx}
\usepackage{tikz}
\usepackage[fleqn]{amsmath}
\usepackage{amssymb}
\usepackage{amsthm}
\usepackage{amsfonts}
\usepackage{bbm}
\usepackage{comment}
\usepackage{xcolor}
\definecolor{red}{rgb}{0,0.1,5}

\begin{document}

\newtheorem*{problem}{Problem}
\newtheorem{theorem}{Theorem}
\newtheorem{corollary}{Corollary}
\newtheorem{lemma}{Lemma}
\newtheorem{proposition}{Proposition}
\theoremstyle{definition}
\newtheorem*{definition}{Definition}
\newtheorem*{remark}{Remark}
\newtheorem*{example}{Example}
\numberwithin{equation}{section}
\everymath{\displaystyle}
\newcommand{\uno}{\mathbbm{1}}
\newcommand{\alocc}[2]{|\!|#1|\!|_{#2}}
\newcommand{\occ}[2]{|#1|_{#2}}
\newcommand{\eps}{\varepsilon}
\newcommand{\wh}[1]{\widehat{#1}}

\newcommand{\alert}[1]{\color{red}{#1}\color{black}}
\newcommand{\note}[1]{\color{red}{#1}\color{black}}

\title{Insertion in constructed normal numbers}

\author{Ver{\'o}nica\  Becher\\
Universidad de Buenos Aires\\
vbecher@dc.uba.ar}

\date{November 4, 2021}
 
\maketitle

\vspace{-9cm}
\noindent
{\tt This version  does not include  Corollary 1 --V.B.}
\vspace{8cm}
\begin{abstract}
Defined by Borel, a real number  is normal to an integer base~$b$,
 greater than or equal to~$2$, if  in its base-$b$ expansion every block of digits 
occurs with the same limiting frequency as every other block of the same length.   
We consider the   problem of {\em insertion} in constructed base-$b$ normal
expansions  to obtain normality to base~$(b+1)$.
\end{abstract}
\bigskip

\section{Problem description and statement of results}

Defined by \'Emile Borel, a real number  is normal to an integer base $b$, 
greater than or equal to~$2$,
if in its base-$b$ expansion every block of digits 
occurs with the same limiting frequency as every other block of the same length.
Equivalently,
a real number $x$ is {normal to base $b$} if  the fractional parts of 
$x, bx, b^2x, \ldots$
are uniformly distributed modulo~$1$ in the unit interval.

There are many ways to modify normal numbers  preserving normality to a given base.
A major  result is Wall's  theorem~\cite{Wall}  showing that the  subsequences of 
a base-$b$ expansion
along   arithmetic progressions preserve normality,
 crowned by  Kamae and Weiss'~\cite{KamaeWeiss}  
complete characterization of the  subsequences that preserve normality. 
Other normality preserving operations are  addition by some  numbers~\cite{Volkmann,Aistleitner,Rauzy},
multiplication by a rational~\cite{Wall},  transformations by some finite automata~\cite{CO} and there are more.

Another form of modification transfers normality from  base~$b$ to normality to base~$(b-1)$:
Vandehey~\cite[Theorem 1.2]{vandehey} proved  that
 the subsequence   of  a  base-$b$ normal expansion 
formed by all the digits different from $(b-1)$ is  normal to base~$(b-1)$.
This is,  indeed, the  removal 
 from a normal  base-$b$ expansion of all the instances of the digit~$(b-1)$.

Here  we consider the dual, the problem of transferring normality from base~$b$ to base~$(b+1)$.

\begin{problem}
How  to  {\it insert} digits along a normal base-$b$ expansion 
 so that the resulting expansion is normal to base~$(b+1)$?
\end{problem}

There are two versions of the insertion  problem:
\begin{itemize}
\item[--] when insertion freely  uses all the digits in base $(b+1)$,
\item[--] when  insertion is limited  just to the new digit. 
\end{itemize}


In the present work we tackle the free  insertion  
problem on  a class  of  {\em constructed} normal  numbers.
We give an effective construction
that  controls the distance  between  each occurrence of the new digit    and the next.
An   effective construction is  a prescription on how to  perform 
the insertion while reading the input sequence from left to right.

Since we look at  normality to just one base at a time, 
instead of  fractional expansions of real numbers we deal  with sequences of symbols in a given alphabet 
and we talk about normality to that alphabet.
We state the results as transferring normality from an  alphabet  $A$  to  alphabet 
$\widehat A=A\cup\{\sigma\}$ with~$\sigma$~not in~$A$.

We consider sequences that  are the  concatenation of
 {\em perfect necklaces} over alphabet~$A$ of linearly increasing order 
and the resulting sequence is also 
a concatenation of  perfect necklaces of linearly increasing order but over alphabet~$\widehat A$.
Perfect necklaces were  introduced in~\cite{ABFY}. 
They are a variant of the classical  de Bruijn sequences.
The  concatenation  of perfect necklaces of linearly increasing order
 is a normal sequence (this is proved in Proposition~\ref{prop:perfect-normal}).
We prove the following.

\begin{theorem}\label{thm:free}
Let alphabets $A$ and   $\widehat A=A\cup\{\sigma\}$ with $\sigma$ not in~$A$.
Let $v\in A^\omega$ be  a concatenation of $(n,n)$-perfect necklaces   over alphabet~$A$ 
for $n=1,2, \ldots$
Then, there is an effective construction  of 
$\widehat v\in \widehat A^\omega$  normal to alphabet $\widehat A$ 
such that $v$~is a subsequence of $\widehat v$
and $\widehat v$  is the concatenation  of  $(n,n)$-perfect necklaces 
over alphabet~$\widehat A$  for $n=1,2, \ldots$
And for every integer~$N$ greater than~$|A|$, 
in between  the occurrences of the symbol  $\sigma$ in $\widehat v$ 
 just before and just after position~$N$ 
there are at most   $2|A|+\log_{|\widehat A|}(N)$ symbols.
\end{theorem}

The one symbol insertion problem  has already an adroit solution  on  arbitrary normal sequences,
given by Zylber in~\cite{Zylber}.

\begin{theorem}[\protect{Zylber~\cite[Theorem 1]{Zylber}}] \label{thm:zylber}
Let alphabets $A$ and   $\widehat A=A\cup\{\sigma\}$ with $\sigma$ not in~$A$.
\linebreak
Let  $v\in A^\omega$ be  normal to alphabet $A$.
Then, there exists  $\widehat v\in \widehat A^\omega$ normal to alphabet $\widehat A$
 such that $r({\widehat v}) = v$, where $r$ is the retract that  removes   all the instances of  the symbol $\sigma$.
\end{theorem}

The construction  in the proof of Theorem \ref{thm:zylber} is not effective in general.
It becomes  effective   when the input  sequence $v\in A^\omega$ satisfies the following condition.
First consider  the simple discrepancy of a sequence $(x_n)_{n\geq 1}$
of real numbers in the unit interval 
 with respect to an interval $I$,
\[
d_{N,I}( (x_n)_{n\geq 1})=\Big|\frac{1}{N}\#\{j:1\leq j\leq N, x_n\in I\}- |I|\Big|
\]
Let $b$ be the cardinality of the alphabet $A$ 
and let  $x$ be the real number denoted by the input sequence $v\in A^\omega$.
The method in the proof of Theorem \ref{thm:zylber} becomes effective 
when there is a computable upper bound of the 
 simple discrepancies $d_{N,I}\big( ((b^\ell)^n x \mod 1)_{n\geq 1} \big)$ 
 for infinitely many integer values $\ell$, 
and for every interval $I$ of the form $(a/b^\ell, (a+1)/b^{\ell})$, 
with $0\leq a< b^\ell-1$.
In particular, it follows from Proposition \ref{prop:aligned} 
that the concatenation of $n$-ordered necklaces for $n=1,2, \ldots$ 
 satisfies the  condition (the $n$-ordered necklace  is the $(n,n)$-perfect necklace  given by 
  the concatenation  of all words of length~$n$ in lexicographic order).

It  remains to study  how to compare the discrepancy of ~$(b^n x \mod 1)_{n\geq 0}$ 
and  the discrepancy of $((b+1)^n y \mod 1)_{n\geq 0}$
when the base-$(b+1)$ expansion of $y$ results from insertion
  in  the base-$b$ expansion of a normal number~$x$. 
It may be possible to obtain metric results  similar to those obtained 
by Fukuyama and Hiroshima~\cite{Fukuyama} for some subsequences  of~$(b^n x\mod 1)_{n\geq 0}$. 
\medskip

This document is organized as follows:
Section~\ref{sec:perfect} presents  the basics of perfect necklaces
and Section~\ref{sec:free}  solves   the free insertion problem 
on the concatenation of  perfect necklaces.

\section{Perfect necklaces and nested perfect necklaces}\label{sec:perfect}

\subsection{Perfect necklaces}

This section is based on~\cite{ABFY}.
A  word  is a finite sequence of symbols in a given alphabet.
For a finite  alphabet~$A$, we write~$|A|$ for its cardinality,
$A^n$ for the set of all words of length $n$,  $A^*$ for the set of all  words
and $A^\omega$ for the set of all infinite sequences.
The positions in words and in sequences are numbered starting at~$1$.
We write $v[i]$  for the symbol at position~$i$ and 
we write $v[i,j]$ for   the symbols  of $v$ from position~$i$ to position~$j$. 
The length of a word $v$ is $|v|$.

Let $\theta: A^* \to A^*$ be the \emph{rotation} operator, 
 $(\theta v)[i]=v[(i+1) \mod |v|])$, position~$i$  between~$1$~and the length of~$v$.
We let~$\theta^n$ denote  the application of the rotation $n$~times. 
A circular word or necklace is the equivalence class of a word under rotations.
To denote a necklace we write $[w]$ where $w$~is any of the words in the equivalence class. 
For instance, $[000]$  contains  a single word $000$ because for every $i$, $\theta^i(000)=000$
and $[110]$ contains three words $\theta^0(110)=110$, $\theta^1(110)=101$ and $\theta^2(110)=011$.

\begin{definition}
A necklace  is  $(n,k)$-perfect   if 
 each word of length $n$ occurs $k$ many times at
 positions different modulo $k$, for any convention of the starting point.
\end{definition}

Thus, each $(n,k)$-perfect necklace has length $k|A|^n$.
Perfect necklaces are a variant of  de Bruijn sequences. 
Recall that a de Bruijn sequence of order $n$ over alphabet $A$ is 
a necklace of  length~$|A|^n$  and each word of length $n$
 occurs in it exactly once. Then,
$(n, 1)$-perfect necklaces coincide with the de Bruijn sequences of order~$n$. 
\medskip

\noindent
Consider alphabet  $A=\{0,1\}$.  The following are  $(2,2)$-perfect necklaces, 
\[
[00\ 01\ 10\ 11]  \text{ and }  [00\ 10\ 01\ 11].
\]
This is a $(3,3)$-perfect necklace  
\[
[000\;110\;101\;111\;001\;010\;011\;100].
\]
The following are not $(n,n)$-perfect
\[
[00\;01\;11\;10] \text{ and }
[000\;101 \;110\;111\;010\;001\;011\;100].
\]

\begin{definition}\label{def:ordered}
For an alphabet $A$  and  a positive integer $n$, the $n$-\emph{ordered necklace}  
  is  the concatenation of all words  of length~$n$ in lexicographic order.
\end{definition}
\noindent
These are the $n$-ordered necklaces over alphabet  $A=\{0,1\}$ for $n=1, 2,3$,
\[ 
[0 1],  \ \ \  [00\ 01\ 10\ 11], \ \ \   [000\ 001\ 010\ 011\  100\ 101\  110\  111]
\]
Every $n$-ordered necklace is $(n,n)$-perfect.
Inexplicably, this  was not observed  by   Barbier~\cite{Barbier2,Barbier1} 
nor by Champernowne~\cite{Champernowne}.

\begin{remark}[\protect{\cite[Theorem 5]{ABFY}}] 
Identify words of length $n$ over  alphabet $A$   with the integers $0$ to~$|A|^n-1$.
Let $r$  coprime with $|A|$.
The concatenation of words  corresponding to the arithmetic sequence 
$0,r,2r,..., (|A|^n -1)r$  yields a $(n,n)$-perfect necklace.
By taking $r=1$ we obtain that $n$-ordered necklaces are $(n,n)$-perfect.
\end{remark}

\begin{proposition} \label{prop:sigma}
In the $n$-ordered necklace over alphabet $A$,
for each symbol $a\in A$, 
between one occurrence of $\sigma$ and the next there are at most   $n |A|-1$ symbols.
\end{proposition}
\begin{proof}
The $n$-ordered necklace is the concatenation of all words of length $n$ in lexicographical order.
Consider   $|A|+1$ many consecutive of these words, $u_1, \ldots u_{|A|+1}$.
Observe that  the last 
symbol in~$u_1$ is necessarily the same as the last symbol in~$u_{|A|+1}$.
Let~$a$ be that symbol. In between these  two occurrences of $a$   there are  $n|A|-1$ symbols.
For some choices of  $u_1, \ldots u_{|A|+1}$
these are the only two occurrences of $a$ in these words.
All the other cases yield  a smaller number of symbols between two occurrences.
\end{proof}

A particular class of perfect necklaces,  called  {\em nested perfect necklaces},
were introduced  in~\cite{BC2019}, generalizing a construction given by M. Levin in 
\cite[Theorem 2]{Levin1999}.
A $(n,k)$-perfect necklace over  alphabet  $A$ is  {\em nested}
if $n=1$ or 
it is the concatenation of $|A|$\ nested   $(n-1,k)$-perfect necklaces.

\noindent
For example,  the following is a nested $(2,2)$-perfect necklace over alphabet $A=\{0,1\}$,
\[
[\underbrace{\ \ 0011}_{\text{\tiny (1,2)-perfect}} \underbrace{\ 0110\ \ }_{\text{\tiny (1,2)-perfect}}]
\]
\noindent
Each of these  $8$   are  $(1,4)$-perfect necklaces.
\begin{gather*}
  [00001111]\ \ \  [01011010] \\
  [00111100]\ \ \ [01101001] \\
  [00011110]\ \ \ [01001011] \\
  [00101101]\ \ \ [01111000]
\end{gather*}

The concatenation in each row yields a $(2,4)$-perfect necklace.

The concatenation of the first two rows   yields a nested $(3,4)$-perfect necklace.  

The concatenation of the last two rows   yields a  nested $(3,4)$-perfect necklace.  

The concatenation of all  rows  yields a nested $(4,4)$-perfect necklace.  
\medskip
\\
The $n$-ordered necklaces are perfect  but {not nested},
for example for  $A=\{0,1,\sigma\}$ and $n=2$,
\[
[\underbrace{00\ 01 \ 0\sigma}_{\text{\tiny not (1,2)-perfect}}\  \underbrace{10\ 11\ 1\sigma}_{\text{\tiny not (1,2)-perfect}}\ \underbrace{\sigma0\ \sigma1\ \sigma\sigma}_{\text{\tiny not (1,2)-perfect}}]
\]

\subsection{Perfect necklaces as Eulerian cycles in astute graphs}

The $(n,k)$-perfect necklaces  are characterized  with  Eulerian cycles in the so called 
{\em astute graphs}.

\begin{definition}\label{def:astute}
The  astute graph $G_A(n,k)$ is a pair $(V,E)$ where 

$V=\{(w,m):  w \in A^n, \ m\in \{0, \ldots k-1\}\}$ and 

$E=\{ \big((w,m), (w',m') \big):  w[2..n]=w'[1.. n-1], \ m'=(m +1)\mod k\}$.
\end{definition}

Thus, $G_A(n,k)$ has $k |A|^n$ vertices and $k|A|^{n+1}$ edges.
It  is Eulerian because it is  strongly regular 
(all vertices have in-degree  and out-degree equal to $|A|$)
  and strongly connected (every vertex is reachable from every other vertex).
Notice that $G_A(n,1)$ is the de Bruijn graph of words of length $n$ over alphabet~$A$.

\begin{proposition}[\protect{\cite[Corollary 14]{ABFY}}] \label{prop:Eulerian}
Each $(n,k)$-perfect necklace over alphabet $A$  can  be constructed as an Eulerian cycle in~$G_A(n-1,k)$.
\end{proposition}

In some cases several  Eulerian cycles in $G_A(n-1,k)$
yield the same  $(n,k)$-perfect necklace,
 this happens when there is a period inside a cycle.

\begin{remark}[\protect{\cite[Theorem 20]{ABFY}}]
The~number  of $(n,k)$-perfect necklaces over a $b$-symbol alphabet~is 
 \[
\frac{1}{k} \sum_{d_{b,k}|j|k} e(j) \phi(k/j),
\] \vspace*{-0.25cm}
where
\begin{itemize}
\item[-]  
 $d_{b,k} = \prod p_i^{\alpha_i}$, 
 such that 
 $\{ p_i \}$  is the set of primes that divide both $b$~and~$k$, 
and $\alpha_i$ is  the exponent of $p_i$ in the factorization of~$k$,

\item[-]
 $e(j)= (b!)^{j b^{n-1}}b^{- n}$ is the number  of Eulerian cycles  in  $G_A(n-1,j)$  where $|A|=b$, 

\item[-] 
$\phi$ is Euler's totient function,
  $\phi(m)$ counts the positive integers 
less than or equal to $m$ that are relatively prime to~$m$.
\end{itemize}
\end{remark}

\begin{remark}[\protect{\cite[Theorem 2]{BC2019}}]
For each $d=0,1,2,\ldots$  there are $ 2^{2^{d+1}-1}$ binary
nested $(2^d,2^d)$-perfect necklaces.
\end{remark}

\subsection{Counting aligned and non-aligned occurrences of words}

For  the number of occurrences of a word $u$ in a word $v$ at any position we write $\occ{v}{u}$,
\[
\occ{v}{u} = |\{i : v[i,i+|u|-1] = u\}|.
\]
For example $\occ{00010}{00}=3$.
We are intersted in counting occurences of a word $u$ in a word $v$ when $[v]$ is a perfect necklace.
\begin{proposition}
If  $[v]$ is a $(n,k)$-perfect necklace over alphabet $A$ 
then for every word $u$  of length at most~$n$,
\[
    k |A|^{|v|-|u|}  - |u| + 1 \leq \occ{v}{u} \leq k |A|^{|v|-|u|}.
\]
\end{proposition}

Recall that the positions in a  a words are numbered starting at~$1$.
 Given two words $v$ and $u$, we write $\alocc{v}{u}$ for  the number of occurrences of  
$u$ at the positions of~$v$  congruent to~$1$ modulo  the length of~$u$, 
that we call  {\em aligned occurrences},
\[
\alocc{v}{u} 
= \big| \{i : v[i,i+|u|-1] = u \text{ and } i \equiv 1 \mod{|u|}\} \big|.
\]
For example, $\alocc{00000}{00}=2$ and $\alocc{1001}{00}=0$.
The relation between  $\occ{v}{u}$  and  $\alocc{v}{u}$ is as follows,
\[
\occ{v}{u} = \sum_{i=0}^{|u|-1} \alocc{v[1+i,|v|]}{u}.
\]
So, for any single symbol~$a$ in the alphabet $A$, 
$
\occ{v}{a}=\alocc{v}{a}.
$

\begin{proposition}\label{prop:aligned}\
\begin{enumerate}
\item
If $[v]$ is  $(n,n)$-perfect  over alphabet $A$, then for every $u$ 
of length $\ell$ where $\ell$ divides $n$,
\[
|A|^{n-\ell} n/\ell -1 \leq  \alocc{v}{u} \leq |A|^{n-\ell} n/\ell.
\]

\item  If $[v]$ is the ordered $(n,n)$-perfect necklace then 
for every $u$  of length $\ell$, where $\ell$ divides $n$,
and for any position $t$ in $v$,
\[
|A|^{-\ell} t/\ell -O(t/n) \leq  \alocc{v[1,t]}{u} \leq |A|^{-\ell} t/\ell + O(t/n).
\]
\end{enumerate}
\end{proposition}

\begin{proof}
1. To count the number of occurrences of   $u$  of length $\ell$ 
in $[v]$, with $1\leq \ell\leq n$,
 we  count how many times $u$ 
occurs at the beginning of  a word of length~$n$.
There are $|A|^{n-\ell}$ many different words of length $n$ that start with~$u$,
and each occurs occurs $n$ times in $[v]$  at positions that are different modulo~$n$.
However, to count occurrences in $v$  we can not regard $v$ circularly,
so  one occurrence of $u$ may be missing . Thus, 
\[
|A|^{n-\ell} n/\ell -1 \leq  \alocc{v}{u} \leq |A|^{n-\ell} n/\ell.
\]

2. If   $[v]$ is $(n,n)$-perfect ordered necklace, so it is the concatenation of
blocks $B_0 \ldots B_{|A|^n-1}$, where each block $B$ has the form 
\[
p_{n-1}p_{n-2}...p_{-1}p_0.
\]
Suppose the $t$-th digit of $v$  occurs within
$p_{n-1}p_{n-2}...p_1p_0$
then
\[
t=n\sum_{j=0}^{n-1} p_j |A|^{j} +\theta n,
\]
for $0< \theta \leq 1$.
Assume $u$ has length $\ell$ less than or equal to $n$.
Let $g_{n,k}(v,t,u)$ denote the number of times that $u$ occurs undivided in the 
first $t$ digits of $v$ with the first digit of $u$ as the $k$-th
 digit of a block $B$ in~$v$.
With $u$  fixed in a position in the block, 
we may choose the last $ n-\ell-k+1$ digits of the block in 
$
|A|^{n-\ell-k+1}
$ many ways.
Having chosen these, in order to ensure that the block
lies as required in  $v[1,t]$,
\[
v[1,t]=\underbrace{0 ... 0}_{\text{block } B_0}\ \underbrace{0 ... 0 1}_{\text{block }  B_1} \ldots
\underbrace{
\overbrace{p_{n-1}p_{n-2}\ldots p_{n-{(k-1)}}}^{\text{first }k-1 \text{ digits of the block}}\
 \overbrace{u_1 u_2 ... u_\ell}^{\text{digits of }u} \
\overbrace{\ldots}^{n-k-\ell+1\text{ digits}}}_{\text{ block } B}
\]
we shall be able to choose the first $k-1$ digits of the block $B$
in this number of ways:
\[
\sum_{j=n-k+1}^{n-1} p_j |A|^{j+k-n-1}\ \  \text{ or  } \ \
\sum_{j=n-k+1}^{n-1} p_j |A|^{j+k-n-1}  +1.
\]
Then, if $k\leq  n-\ell+1$,
\begin{align*}
g_{n,k}(v,t,u)&= |A|^{n-\ell-k+1} 
 \sum_{j=n-k+1}^{n-1} p_j |A|^{j+k-n-1} +\theta', 
\\
&= |A|^{-\ell} 
\sum_{j=n-k+1}^{n-1} p_j |A|^{j}  +\theta' |A|^{n-k+1}, \text{ where $0\leq\theta'\leq 1$}.
\end{align*}
To obtain a lower bound of $\alocc{v[1,t]}{u}$ we sum 
 $g_{n,k}(v,t,u) $ for every   $k$  such that $(k\mod  \ell)=K$,
where   $K$ is $( n\mod \ell)$,
\begin{align*}
\alocc{v[1,t]}{u}&\geq
\sum_{\substack{k: (k\mod \ell )=K\\1\leq k\leq n-\ell+1}}  g_{n,k}(v,t,u) 
\\&=
\sum_{\substack{k: (k\mod \ell )=K\\1\leq k\leq n-\ell+1}}  |A|^{-\ell} 
\sum_{j=n-k+1}^{n-1} p_j |A|^{j} 
\\
&=|A|^{-\ell}  \sum_{j=\ell+K}^{n-1} \lfloor j/\ell\rfloor  p_j |A|^{j} 
\\
&\geq  |A|^{-\ell}  /\ell \Big( \sum_{j=\ell+K}^{n-1} n  p_j |A|^{j}  -
\sum_{j=\ell+\eta}^{n-1} (n-j)  p_j |A|^{j} \Big)
\\
&=  |A|^{-\ell}  t/\ell  -O(t/n).
\end{align*}
Finally notice that  the aligned occurrences of $u$ can occur at most once in between every 
two blocks of~$v$, 
\begin{align*}
\alocc{v[1,t]}{u}&\leq
\Big(\sum_{\substack{k: (k\mod \ell )=K\\1\leq k\leq n-\ell+1}}  g_{n,k}(v,t)  \Big)+O(t/n)
\\&=
\Big(\sum_{\substack{k: (k\mod \ell )=K\\1\leq k\leq n-\ell+1}}  |A|^{-\ell} 
\sum_{j=n-k+1}^{n-1} p_j |A|^{j}  \Big) +O(t/n)
\\
&=|A|^{-\ell}  \Big(\sum_{j=\ell+\eta}^{n-1} \lceil j/\ell\rceil  p_j|A|^{j} \Big)+O(t/n)
\\
&\leq  |A|^{-\ell}  /\ell \Big( \sum_{j=\ell+K}^{n-1} n  p_j |A|^{j}  -
\sum_{j=\ell+\eta}^{n-1} (n-j)  p_j |A|^{j} \Big)  +O(t/n)
\\
&=  |A|^{-\ell}  t/\ell  +O(t/n).
\end{align*}
\end{proof}

\subsection{From perfect necklaces to normal sequences }

We show that the concatenation of perfect necklaces of linearly increasing order is normal.
To prove it we use  Piatetski-Shapiro's  theorem~\cite{PS,MS,Bugeaud2012}. 

\begin{proposition}[Piatetski-Shapiro theorem]  \label{prop:piatetski-shapiro}	\label{hotspot}
The sequence  $v\in A^\omega$ is normal to alphabet $A$ if and only if there is positive constant~$C$ such that for all words~$u$,
\[
	\limsup_{n\rightarrow\infty}   \frac{\occ{v[1,n]}{u}}{n}  \leq C |A|^{-|u|}.
\]
\end{proposition}

\begin{proposition} \label{prop:perfect-normal}
The concatenation of $(n,k)$-perfect necklaces over alphabet $A$, 
for $n=1,2, \dots$ and $k_n$ a linear function of $n$,
 is normal to alphabet $A$.
\end{proposition}
\begin{proof} 
Let $M(0)=0$ and for $m\geq 1$, $M(m)=\sum_{i=1 }^m k_i |A|^i $. 
\\
Fix $N$.
Let $m$ be such that  $M(m-1)<N\leq M(m)$. 
Then, for every $u$ of length $\ell$,
\begin{align*}
\occ{v[1,N]}{u}
&\leq\occ{v[1,M(m)]}{u}
\\
&\leq \sum_{i=1}^{m}   \occ{v[M(i-1)+1, M(i) ]}{u} + \ell-1
\\
&\leq \sum_{i=1}^{m}   k_i |A|^{i-\ell}  + \ell-1
\\
&\leq k_m |A|^{m+1} |A|^{-\ell} + m\ell.
\end{align*}
Since  $k_n$ is linear in $n$, for every $n$ we have  $k_{n}/k_{n-1}$ is a constant $c$.\\
Then, using $M(m)\leq k_m|A|^{m+1}$  and $M(m-1)\geq k_{m-1}|A|^{m-1}$, 
\begin{align*}
	\limsup_{N\rightarrow\infty}   \frac{\occ{v[1,N]}{u}}{N}  
&\leq \limsup_{m\rightarrow\infty} \frac{\occ{v[1,M(m)]}{u}}{M(m-1)}
\\
&\leq\limsup_{m\rightarrow\infty}  \frac{k_m |A|^{m+1} |A|^{-\ell} + m\ell}{ k_{m-1}|A|^{m-1}}
\\
& \leq c|A|^{2} |A|^{-\ell}.
\end{align*}
Piatetski-Shapiro  theorem (Proposition~\ref{hotspot})
 holds with~$C=c|A|^2$ and  $v$ is normal to alphabet~$A$.
\end{proof}

Actually  Proposition \ref{prop:perfect-normal} holds  for 
$k_n$ being polynomial in $n$ and the same proof applies. 
The concatenation of nested perfect necklaces of exponentially 
increasing order also yields normal sequences,
but this can not be  proved using Proposition \ref{prop:perfect-normal}.
By Wall's thesis~\cite{Wall},  normal numbers are exactly those real numbers $x$   for which 
 $(b^n x\mod 1)_{n\geq 1}$ is uniformly distributed, 
which means that  the discrepancy of the first $N$ terms
\[
D_{N}\left((b^n x \mod 1\right)_{n\geq 0})=\!\! 
\sup_{\gamma\in[0,1)} \left\vert  \frac{1}{N} \left| \{n\leq N  : ( b^n  x \mod 1) \} <\gamma\}\right|   - \gamma \right\vert
\]
 goes to $0$ as $N$ goes to infinity.
For sequences of the form $(b^n x \mod 1)_{n\geq 1}$
 the smallest known discrepancy of the first $N$ terms is $O((\log N)^2/N)$, 
see~\cite{Levin1999,Bugeaud2012}.
Expansions made of nested perfect necklaces of exponentially 
increasing order   yield  real numbers $x$ with this property.

\begin{remark}[\protect{\cite[Theorem 1]{BC2019}}] \label{thm:nested-normal}
 Let $b$ a prime number.
The  base-$b$ expansion of the number defined by M.Levin 
using  Pascal triangle matrix modulo~$2$ 
is the concatenation of  nested $(2^d,2^d)$-perfect necklaces
  for  $d=0,1,2,\ldots$.  
And  for every   number~$x$ whose base-$b$  expansion is the concatenation of nested
  $(2^d,2^d)$-perfect necklaces for  $d = 0,1,2\ldots$, $D_N((b^n x \mod 1)_{n\geq 0})$ is $O((\log N)^2/N)$.
\end{remark}

In general, the   discrepancy associated to  the concatenation of 
$(n,k)$-perfect necklaces has not been studied.
One exception is the discrepancy associated to the concatenation $n$-ordered necklaces
which is exactly the discrepancy associated to Champernowne's sequence~\cite{Champernowne}
proved in~\cite{schiffer}, see also \cite{Bugeaud2012,DrmotaTichy1997}.

\begin{remark}[\protect{\cite[Theorem 1]{schiffer}}] \label{remark:schiffer}
The number $x$ whose base $b$ expansion is
the concatenation of the $n$-ordered necklaces for $n=1, 2, \ldots$
$D_N((b^n x \mod 1)_{n\geq 0})$ is $O(1/(\log N))$.
\end{remark}

\section{Free insertion} \label{sec:free}

\subsection{Tools to prove Theorem~\ref{thm:free}} \label{sec:tools-free}

Consider alphabets $A$ and $\widehat A=A\cup\{\sigma\} $ for $\sigma$ not in $A$.
Since the length and lexicographic order  on words over alphabet $A$
respects  the length and lexicographic order  on words over~$\widehat A$,
by inserting suitable symbols  in  suitable positions  in each $n$-ordered necklace over $A$
we obtain each $n$-ordered necklace over~$\widehat A$.
For example,  for $A=\{0,1\}$ and $\widehat A=\{0,1,\sigma\}$,
\[
\footnotesize
\begin{array}{lll}
0\, 1\color{red}_{_{\tiny \blacktriangle} }\!\! \color{black}   \ \ 
00\ 01\! \color{red}_{_{\tiny \blacktriangle} }\!\! \color{black}10\ 11\  \ \ \ 
000\ 001\! \color{red}_{_{\tiny \blacktriangle} }\!\! \color{black} 
010\ 011\! \color{red}_{_{\tiny \blacktriangle} }\!\! \color{black} 
100\ 101\! \color{red}_{_{\tiny \blacktriangle} }\!\! \color{black}
 110\ 111\ 0000\ 0001\ . . . 
\\\\
0\, 1\, \alert{\sigma} \ \ 
00\ 01\ \alert{0\sigma}\ 10\ 11\  \alert{1\sigma \ \sigma 0\ \sigma 1\ \sigma \sigma }\ \ 
000\ 001\ \alert{00\sigma}\ 010\ 011\  \alert{01\sigma\ 0\sigma 0\ 0\sigma 1\ 0\sigma \sigma}\
100\ 101\ \alert{10\sigma }\ 110\ 111\ \alert{11\sigma \ 1\sigma 0\ 1\sigma 1\ 1\sigma \sigma  }\\
\ \ \ \ \ \ \ \  \alert{ \sigma 00\ \sigma 01\ \sigma 0\sigma \ \sigma 10\ \sigma 11\ \sigma 1\sigma \ \sigma \sigma 0\ \sigma \sigma 1\ \sigma \sigma \sigma }\ \ \ \  0000\ 0001 ...
\end{array}
\]
Much more is true: for  any $(n,k)$-perfect necklace over alphabet  $A$
there is  an $(n,k)$-perfect necklace over $\widehat A$ such that 
the first is a subsequence of the second. 
This is immediate from the graph theoretical  characterization 
$(n,k)$-perfect necklaces  as Eulerian cycles on  astute graphs: 
$G_A(k,n-1)$  is a subgraph of $G_{\widehat A}(k,n-1)$, 
 and any cycle in an Eulerian graph can be embedded into a full Eulerian cycle. 
This  can be constructed with Hierholzer's algorithm 
for joining cycles together to create an Eulerian cycle of a graph. 
However this method does not  guarantee  that 
in the resulting $(n,k)$-perfect necklace over alphabet $\widehat A$,
there will be a small gap  between one occurrence of the symbol~$\sigma$ and the next.

The following lemma gives a  method to insert symbols in a $(n,n)$-perfect necklace 
ensuring   a {\em  small gap condition}.
The lemma   extends the work done for  de Bruijn sequences  in~\cite[Theorem 1]{BCortes}.

\begin{lemma}[Main lemma]\label{thm:2to3}
Assume alphabets $A$ and $\widehat A=A\cup\{\sigma\} $ for $\sigma$ not in $A$.
For every $(n,n)$-perfect necklace~$[v]$   over  alphabet $A$
there is a $(n,n)$-perfect necklace~$[\widehat v]$ over alphabet $\widehat A$
such that $v$  is a  subsequence of $\widehat v$.  
Moreover, for each such $[v]$ there is  $[\widehat v]$ satisfying   that 
in between any  occurrence of the symbol~$\sigma$ and the next there are
at most~$n+2|A| -2$ other symbols.
\end{lemma}

Notice that  the $n$-ordered necklace  over alphabet   $\widehat A$  
fails the small gap condition required in  the Main Lemma~\ref{thm:2to3}.
For instance,  for $A=\{0,1\}$, $\widehat A=\{0,1,\sigma \}$
 and  $n=2$, there are occurrences of $\sigma$ with more than $n+2|A| -2=4$ symbols in between:
\medskip

$[v]=[\   00\ 01_{\color{red}\blacktriangle}  10\ 11_{\color{red}\blacktriangle}]$  
\medskip

$[\widehat v]=[\underbrace{00\ 01\ \alert{0}}_{5\ symbols}\!\! \alert{\sigma } \underbrace{10 \ 11\ \alert{1}}_{5\ symbols}\!\!\alert{\sigma \ \sigma 0 \ \sigma 1\ \sigma \sigma }]$ 
\medskip
\\
However this other  insertion satisfies the small gap condition:
\medskip

$[\widehat v] =[ \underbrace{0}_{}\alert{\sigma } \ 
\underbrace{\alert{0}0 \ 01}_{4\ symbols}\ 
\alert{\sigma }\!\!\!\!\!
\underbrace{\alert{1}\   10}_{3\ symbols}\ 
\alert{\sigma \sigma }\ \underbrace{\alert{0} 1}_{2\ symbols}\ 
\alert{\sigma \sigma }\  \underbrace{\alert{1}1}_{2+ 1\ symbols}]$.
\medskip

To prove the Main Lemma~\ref{thm:2to3}  we start with a $(n,n)$-perfect necklace $[v]$ over alphabet $A$,
 we consider an Eulerian cycle in $G_{A}(n-1,n)$ that corresponds to $[v]$
and we  extend it to  an Eulerian cycle in  graph
$G_{\widehat A}(n-1,n)$. The $(n,n)$-perfect necklace that describes  this cycle is the wanted $[\widehat v]$. 

Since for every pair of  positive integers $n,k$, 
  $G_A(n,k)$ is a subgraph of $G_{\widehat A}(n,k)$  the following is well defined.

\begin{definition}[Augmenting graph]
The augmenting graph $X_{\widehat A}(n-1,n)$ is the directed graph $(V, E)$ where 
\begin{align*}
V &=\widehat A^{n-1} 
\\
E &=
\left\{ 
\begin{array}{ll}
\big( (au,m),(ub,(m+1)\!\!\!\!\mod n) \big) : &
 u\in \widehat A^{n-2},  a,b\in \widehat A,
\Big( \occ{au}{\sigma}>0 \text{ or } \occ{ub}{\sigma}>0 \Big) 
\\
& m\in\{0,\ldots , n-1\} 
\end{array}
\right\}
\end{align*}
\end{definition}

Each vertex  in $X_{\widehat A}(n-1,n)$  
that is also a vertex in  $G_A(n-1,n)$ 
has exactly one incoming edge and exactly one outgoing edge.
This outcoming edge  is  associated to  new symbol~$\sigma$.

We say that two   cycles are  disjoint  if they have  no common edges. 
To prove the Main Lemma~\ref{thm:2to3} we  construct an 
 Eulerian cycle in $G_{\widehat A}(n-1,n)$  by joining 
the given  Eulerian cycle in $G_A(n-1,n)$ with disjoint cycles 
of the augmenting graph $X_{\widehat A}(n-1,n)$ that we   call {\em petals}.
These petals must exhaust the augmenting graph~$X_{\widehat A}(n-1,n)$.
Recall that $\theta$ is the rotation operation on words that shifts one position to the right.

\begin{definition}[Necklaces on pairs $(v,m)$]
Assume alphabet  $\widehat A$ and a positive integer $n$.
For $u\in \widehat A^n$ and $m$  between $0$ and $n-1$, the necklece $[(v,m)]$ is :
\[
 [(u,m)]=\big\{(u,m), (\theta(u), (m+1) \!\!\mod n),  (\theta^2(u), (m+2) \!\!\mod n), \ldots, (\theta^n(u), (m+n-1) \!\!\mod n)\big\},
\]
\end{definition}

\begin{proposition}
The set of edges in $G_{\widehat A}(n-1,n)$ can be partitioned in disjoint 
simple cycles identified by the necklaces of pairs of words length~$n-1$
 in alphabet $\widehat A$ and $n$ congruence classes.
\end{proposition}
\begin{proof}
Let $u\in \widehat A^{n-1}$, let $a\in \widehat A$ and  let $m\in\{0, \ldots n-1\}$.
Consider the elements in $[(ua,m)]$ and the sequence of  edges
\[
(ua,m) \to (\theta({ua}), (m+1)\!\!\mod n)\to (\theta^2({ua}),(m+2)\!\!\mod n)\to \ldots\to (\theta^{n}({ua}),(m+n)\!\!\!\mod n)
\]
Observe that 
the vertices related by these  edges are pairwise different
except  
\[
(ua,m)= (\theta^{n}({ua}),(m+n)\!\!\!\mod n).
\]
Thus, these $n$ edges  form a simple cycle in $G_{\widehat A}(n-1,n)$.
For each congruence class $m$, the  partition of the set of  words of length~$n$ 
in the equivalence classes given by their  rotations 
determines a partition of the set of edges in $G_{\widehat A}(n-1,n)$ into disjoint simple cycles.
\end{proof}

\begin{definition}[Graph of necklaces]
1. Define  $C_{\widehat A}(n,n)$  as the graph $(V, E)$ where
\begin{align*}
V &=\{ [(u,m)] : u\in \widehat A^n, m=0, \ldots , n-1\}
\\
E&= \{ (x,y):     \mbox{there is } (au,m)\in x \mbox{ and  there is  } (uc, (m+1)\!\!\mod n) \in y, \mbox{ for } a,c\in\widehat A \}
\end{align*}
2. Define  the graph    $\widehat C_{\widehat A}(n,n)$ as the subgraph  of  $C_{\widehat A}(n,n)$
 whose  vertices contain at least one occurrence of the symbol~$\sigma$.
\end{definition}

We are ready to define a petal for each  vertex in $G_A(n-1,n)$.
A petal  is a union of disjoint cycles in $X_{\widehat A}(n-1,n)$  
 which are identified by  the necklaces of length~$n$ that have at least one occurrence of  symbol~$\sigma$.
For this identification  we consider the graph  $\widehat C_{\widehat A}(n,n)$.

\begin{definition}[Petal for vertex in $G_A(n-1,n)$]
A {\em  petal} for  vertex $(u,m)$ in $G_A(n-1,n)$  is a cycle in $ X_{\widehat A}(n-1,n)$
induced by a  subgraph of  $\widehat C_{\widehat A}(n,n)$ that  contains the necklace $[(u\sigma,m)]$.
The petal  for vertex $(u,m)$ in $G_A(n-1,n)$  
starts at  the necklace~$[(u\sigma,m)]$ in $\widehat C_{\widehat A}(n,n)$.
\end{definition}

To exhaust $ X_{\widehat A}(n-1,n)$  we partition it in petals.
For this we define  a {\em Petals tree}.
Recall that a tree  is a directed acyclic graph  with exactly one path from the root to each vertex.

\begin{definition}[Petals tree]
A  Petals tree  for  $\widehat C_{\widehat A}(n,n)$ consists  of  a root~$[r]  $ 
 that branches out in a  subgraph of~$\widehat C_{\widehat A}(n,n)$
including all its vertices.
It  has    has height $n$, 
 the vertices at  distance~$d$ to the root have exactly $d$ occurrences 
of the new symbol~$\sigma$, for $d=0, \ldots, n$. 
The root $[r]$  is a necklace that corresponds to an Eulerian cycle in~$G_A(n-1,n)$.
\end{definition}
There are many Petals trees for  $\widehat C_{\widehat A}(n,n)$, any one is good for our purpose.
A Petals tree  can be obtained by any algorithm that finds a  spanning tree of a graph, 
as  Kruskal's greedy algorithm for the minimal spanning tree,
or it can be constructed using  the classical Breath First search  on~$ \widehat C_{\widehat A}(n,n)$.

We now focus on how to insert the petals  in the given Eulerian cycle 
in $G_A(n-1,n)$
 but satisfying the small  gap condition.
Since each vertex $u $ of $G_A(n-1,n)$ occurs exactly~$|A|$ times in the given Eulerian cycle, 
we have~$|A|$ many possibilities to place the petal for $u$.  
To determine where to place it, we divide the given Eulerian cycle 
in as many  consecutive {\em sections} as the number of vertices in the graph $G_A(n-1,n)$. 
We say that an   Eulerian  cycle is pointed when there is a designated first edge.

\begin{definition}[Section of a cycle]
For a pointed  Eulerian cycle in $G_A(n-1,n)$  given by the  sequence of  edges  
 $e_1, \ldots, e_{n|A|^{n}}$
and  an  integer $j$  such that $0\leq j<n |A|^{n-1}$,
the $j$-th section of the cycle is  the  sequence of the $|A|$ vertices that are 
heads  of $e_{j|A|}, \ldots ,e_{j|A|+|A|-1}$ .
\end{definition}

We would like to choose one vertex from each section to place a petal. 
The problem is that each vertex  occurs  $|A|$ times in the Eulerian cycle  
but not necessarily at $|A|$ different sections.
We pose  a {\em matching} problem.

\begin{definition}[Distribution graph]
Given pointed   Eulerian cycle  in  $G_A(n-1,n)$  
the   Distribution graph $D_A(n-1,n)$ is a $|A|$-regular bipartite graph,
one part consists of the  vertices in  $G_A(n-1,n)$,
the other part consists of the  sections of  the Eulerian cycle.
There is an  edge  from a vertex $u$ in $G_A(n-1,n)$ to a section $j$ if $u$ belongs to the section~$j$.
\end{definition}

A  matching  in  a Distribution graph is a set of edges such that no two edges share a common vertex. 
A vertex is  matched if it is an endpoint of one of the edges in the matching.
 A {\em perfect matching} is a matching that matches all vertices in the graph.

\begin{proposition}\label{prop:matching}
For every  Distribution graph $D_A(n-1,n)$ there is a perfect matching.
\end{proposition}

\begin{proof}
Let $D$ be a finite bipartite graph consisting of 
  two disjoint  sets  of vertices $X$ and $Y$ with
 edges that  connect a vertex in $X$ to a vertex in $Y$.
For a subset $W$ of  $X$, let $N(W)$ be the  set of all vertices in $Y$ adjacent to some element in~$W$. 
Hall's marriage theorem~\cite{hallph} states that there is a matching that entirely 
covers $X$ if and only if for every subset $W$ in $X$, $|W| \leq |N(W)|$.
Consider a Distribution graph $D_A(n-1,n)$ and call   $X$ to   the set of vertices $G_A(n-1,n)$ and  $Y$ to  the set of sections.
For any  $W\subseteq X$ such that $|W| = r$, the sum of the out-degree of these $r$ vertices is~$r|A|$. 
Given that the in-degree for any vertex in $Y$ is~$|A|$, we have that $|N(W)| \geq r$. 
Then, there is a matching that entirely covers $X$. 
Furthermore,  since the number of vertices is equal to the number of sections, 
 $|X| = |Y|$,  the matching is perfect.
\end{proof}

To obtain a perfect matching in a Distribution 
graph we can use any method to compute the maximum flow in a network. 
We define the flow  network by 
adding  two vertices  to the  Distribution graph,  the source and the sink.
Add an edge from the source  to each vertex in $X$ 
and add an edge from each vertex in $Y$ to the sink. 
Assign capacity~$1$ to each of the edges of the flow network. 
The maximum flow of the network is $|X|$.
This flow has the edges of a perfect match.
\medskip

We have the needed tools for the awaiting proof.

\begin{proof}[Proof of  the Main Lemma ~\ref{thm:2to3}]
Assume   $[v]$ is a $(n,n)$-perfect necklace over alphabet $A$.
We construct a $(n,n)$-perfect necklace $[\widehat v]$ over alphabet $\widehat A$.
By Proposition~\ref{prop:Eulerian},
we need to construct an Eulerian cycle in~$G_{\widehat A}(n-1,n)$.
Consider  a pointed Eulerian cycle in $G_A(n-1,n)$  with starting edge determined by~$v$ 
and  consider its $n|A|^{n-1}$ sections.

By Proposition~\ref{prop:matching}  we  
choose one vertex in each section according to a perfect matching.
Fix a Petals tree for $C_{\widehat A}(n,n)$ with root $[v]$. 
The construction traverses the graph $G_{\widehat A}(n-1,n)$ 
until it obtains an Eulerian cycle in this graph.
The idea is to  insert one petal in each section of the Eulerian cycle in $G_{A}(n-1,n)$. 
All the sections are considered, one after the other, until all sections have been considered.
The construction starts  at section~$0$,
and at each step of the construction there is a current section.

The traversal of of $G_{\widehat A}(n-1,n)$ starts at the  starting vertex of 
the pointed  Eulerian cycle in $G_{A}(n-1,n)$ .
Each time an edge is traversed, the current vertex becomes the edge's endpoint.
Let $(u,m)$ be the current vertex.

{\em Case $(u,m)$ is a vertex in $G_A(n-1,n)$}:
If~$(u,m)$ is the chosen vertex in the current section and 
 the petal for~$(u,m)$, 
which starts with~$[u\sigma,m]$ in $\widehat C_{\widehat A}(n,n)$, 
has not been inserted yet then insert it now:
traverse the edge  that adds the symbol~$\sigma$
and continue traversing  the edges in $G_{\widehat A}(n-1,n)$ corresponding to the petal for $(u,m)$.
If the petal for $(u,m)$ has already been inserted
or $(u,m)$ is not a chosen vertex
then  continue with the traversal of edges corresponding to  the current section.
If the current section is exhausted, the next section becomes the current section.

{\em Case  $(u,m)$  is not a vertex in~$G_A(n-1,n)$}:
If   $[u\sigma,m]$ is a child  of the current  vertex in the Petals tree and  
it  has not been inserted yet, then 
traverse the edge  that adds the symbol~$\sigma$
and continue traversing  the edges in $G_{\widehat A}(n-1,n)$
 corresponding to the petal for $(u,m)$.
Otherwise continue with the traversal of the edges corresponding
 to  petal that $(u,m)$ was already part of.

Finally, we prove  that the construction of satisfies the minimal gap condition.
We must prove that  in $[\widehat v]$
in between any occurrence of~$\sigma$  and the next 
there are at most $n+2|A|-2$ symbols.
Obviously each section of the Eulerian cycle in  $G_{A}(n-1,n)$
 has no occurrence of the symbol~$\sigma$.
A petal for a vertex~$(u,m)$ in $G_A(n-1,n)$  
necessarily starts with  the edge that adds the symbol~$\sigma$ right after~$u$,
and this petal corresponds to a path in $\widehat C_{\widehat A}(n,n)$.

Since each section has~$|A|$ edges,  if we place one petal in each section then 
two consecutive  petals are at most $2|A|-1$  edges away. 
Pick a section and let $(u,m)$ be the chosen vertex  and  let $(u',m')$
 be the chosen vertex in the next section.
In case  the   petal  for  $(u,m)$ corresponds just the single vertex  $[u\sigma,m]$ in $\widehat C_{\widehat A}(n,n)$, 
then  it is  a cycle  in $G_{\widehat A}(n-1,n)$ consisting of  exactly~$n$ edges.
So, in between the  occurrence of $\sigma$ in the petal for $(u,m)$ and the first occurrence
 of $\sigma$ in the petal for $(u',m')$  there are at most~$n-1 + 2|A| -1= n + 2|A| -2$ other symbols.
In  case  the petal  for $(u,m)$ consists of   more than one  vertex in $\widehat C_{\widehat A}(n,n)$
then,   before completing the traversal of the $n$ edges  corresponding to  $[u\sigma,m]$,
the construction 
\begin{enumerate}
\item first branches out to another vertex in $\widehat C_{\widehat A}(n,n)$ and

\item then traverses  the corresponding edges   in $G_{\widehat A}(n-1,n)$
that relate vertices having at least one occurrence of the symbol~$\sigma$,

\item and returns  back to $(u,m)$, necessarily from a vertex $(\sigma w, (m-1) \mod n)$, where
$w$ is the  prefix of $u$ of length $n-2$.
\end{enumerate}
So,  in between the last   occurrence of $\sigma$ in the petal for $(u,m)$ 
and the first occurrence
 of $\sigma$ in the petal for $(u',m')$  there are at most~$n-1 + 2|A| -1=n+2|A|-2 $ other symbols.

It remains to argue  what happens inside a petal.
The  vertices in $\widehat C_{\widehat A}(n,n)$ correspond to edges in the 
 $G_{\widehat A}(n-1,n)$  that relate vertices with at
  least one occurrence of the symbol $\sigma$.
This ensures that in between any two successive occurrences of $\sigma$  inside
 a petal  there are at most $n-1$ symbols.

We conclude that in the   traversal of  $G_{\widehat A}(n-1,n)$  
 in between any occurrence of $\sigma$ and the next there 
are at most~$n+ 2|A| -2$ other symbols.
\end{proof}
\medskip

\begin{example}
Consider  $A=\{0,1\}$ and $\widehat A=\{0,1,\sigma \}$. 
Let $[v]$ the  $n$-ordered necklace   for $n=2$,
\[
[v]=[00\ 01\ 10\ 11]
\]  
Fix  an  Eulerian cycle for  $[v]$ in~$G_A(1,2)$.
\medskip

\noindent
Since $G_A(1,2)$ has $4$ vertices, divide it  in $4$ sections:

Section $0$   contains the vertices 
$(0,0 )$ and $(0,1)$

Section $1$   contains the  vertices 
$(0,0)$ and  $(1,1)$.

Section $2$  contains the  vertices 
$(1, 0 )$  and     $(0,1 )$.

Section $3$  contains the  vertices  $(1,0)$ and  $(1,1)$.
\medskip

The following choice gives a perfect match:

Section $0: (0,0)$,  
Section $1: (1,1)$, 
Section $2: (0,1)$, 
Section $4: (1, 0)$.

\newpage


Consider the following Petals tree with root $[r]=[00011011]$,

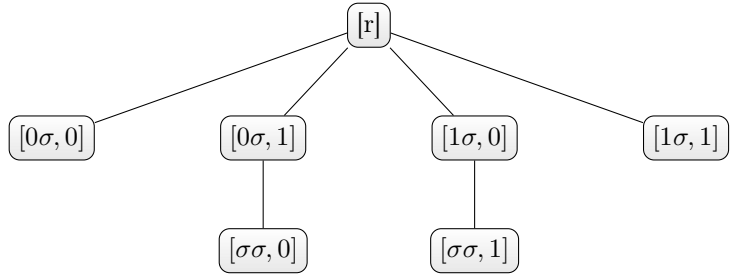
\begin{center}
\begin{tikzpicture}[sibling distance=8em,
  every node/.style = {shape=rectangle, rounded corners,
    draw, align=center,
    top color=white, bottom color=gray!20}]]
  \node {[r]}
    child { node {$[0\sigma ,0]$} }
    child { node {$[0\sigma ,1] $}
               child {node{$[\sigma \sigma ,0] $}}
             }
    child { node {$[1\sigma ,0]$} 
        child {node{$[\sigma \sigma ,1] $}}
            }
    child { node {$[1\sigma,1]$}};
\end{tikzpicture}
\end{center}

This Petals tree has $4$ branches, each one  is a petal for a vertex in $G_A(1,2)$:
\begin{enumerate}
\item The first branch is a petal for $(0,0)$,
It  results in the sequence $\sigma 0$ to be inserted 
(right after the $0$ that appears at an  even position).

\item The second branch is a petal for $(0,1)$.
It is  the join of two vertices in the tree, 
which results in  the  sequence $\sigma \sigma 0$ to be inserted 
 (right after the $0$ that appears at an  odd position).

\item The third branch is a petal for $(1,0)$.
It   results in the   sequence $\sigma \sigma 1$ to be inserted
 (right after the $1$ that appears at an  even position).

\item The fourth branch  is a petal for $(1,1)$.
It results in the   sequence $\sigma 1$ to be inserted 
 (right after the $1$ that appears at an odd position).
\end{enumerate}

The construction follows the Eulerian cycle for $[v]$ in $G_A(1,2)$
and,  in each section it inserts the petal for the chosen vertex for a perfect match, immediately after it.
Thus,\\
 in section $0$ it inserts the edges for $[0\sigma,0]$ at $(0,0)$,\\
 in section $1$ it inserts the edges for $[1\sigma,1]$ at $(1,1)$,\\
 in section $2$ it inserts the edges for $[0\sigma,1]$ and $[\sigma\sigma,0]$ at $(0,1)$,\\
 in section $3$ it inserts the edges for $[1\sigma,0]$and $[\sigma\sigma,0]$ at $(1,0)$.
\\
The result is  the $(1,2)$-perfect necklace $[\widehat v]$ over alphabet $\widehat A$.
The inserted symbols are coloured:
\[
[\widehat v]= [0\color{red}\sigma  \ 0\color{black}0 \ 01\ \color{red}\sigma 1\ \color{black} \ 10\ \color{red}\sigma \sigma \ 0\color{black} 1\ \color{red}\sigma \sigma \ 1\color{black}1].
\]

It satisfies the small gap condition,
because between any occurrence of $\sigma$ and the next there are at most $4$ other 
symbols, which is less that the allowed because  $n=2$, $|A|=2$ and  
$n+2|A| - 2=5$ symbols.
\end{example}

\subsection{Proof of Theorem~\ref{thm:free}} \label{sec:proof-free}

Suppose $v\in A^\omega$ is the concatenation of $(n,n)$-perfect necklaces over alphabet $A$, 
for $n=1,2, \ldots$.
Apply  the Main Lemma~\ref{thm:2to3}  to each of these  $(n,n)$-perfect necklaces over alphabet $A$
and obtain $(n,n)$-perfect necklaces over alphabet $\widehat A$.
By Proposition~\ref{prop:perfect-normal},  their  concatenation 
 is normal to alphabet~$\widehat A$.

Fix  a positive integer $N$.
Recall that the length of a $(n,n)$-perfect necklace over alphabet $\widehat A$ is
$n |\widehat A|^n$.
Let $m$ be such that 
\[
\sum_{i=1}^{m-1} i |\widehat A|^i < N \leq \sum_{i=1}^{m} i |\widehat A|^i.
\]
Therefore, 
$
|\widehat A|^m < N, \text{ hence, } m\leq \log_{|\widehat A|}N.$

Consider  the possibilities for   the occurrences of $\sigma$  in $\widehat v$ 
just before and just after position~$N$.
We need to analyze  two cases.

{\em Case  they are both inside the same $(n,n)$-perfect necklace.}
The Main Lemma~\ref{thm:2to3} proved that the number of symbols in between is less than~$2 |A|+n$,
henceforth less than $ 2 |A|+m$.

{\em Case they are not in  same perfect necklace.}
Notice that in the proof of  the Main Lemma~\ref{thm:2to3}   
the  Eulerian cycle   over alphabet~$A$  is divided in  sections of size $|A|$,
independently of the value of $n$.
Assume  that~$N$ is in the~$(m,m)$-perfect necklace over $\widehat A$.
First suppose that    the occurrence  before position $N$ is in the $(m-1,{m-1})$-perfect necklace.
The construction in Lemma \ref{thm:2to3} ensures that~$\sigma$ occurs
 in the last  $|A|+m$ symbols of this necklace
and the next occurrence of $\sigma$ is  in the first $|A|+1$ symbols 
of the $(m,{m})$-perfect necklace.
Thus, in between these two occurrences of~$\sigma$ there are at most $2|A|+m-1$ symbols.
Now suppose that   the occurrence  after position $N$ is in the $(m+1,{m+1})$-perfect necklace.
Then, there is an occurrence of $\sigma$  in the last  $|A|+m+1$ symbols  of 
the $(m,m)$-perfect necklace and the next occurrence of $\sigma$ is in the first  $|A|+1$ symbols 
of the $(m+1,{m+1})$-perfect necklace.
Therefore, in between the two occurrences of $\sigma$ are at most~$2|A|+m$ symbols. 

Since $m\leq \log_{\widehat A} N$, it follows  that for every $N$ 
 the number of symbols 
in between these occurrences of $\sigma$  before an after position $N$
is at most~$ 2 |A|+\log_{|\widehat A|}N $.
This concludes the proof of Theorem~\ref{thm:free}.

\bibliographystyle{plain}
\bibliography{udt}

\bigskip
\bigskip

\noindent
Ver\'onica Becher\\
 Departmento de  Computaci\'on,   Facultad de Ciencias Exactas y Naturales\\
 Universidad de Buenos Aires \& ICC CONICET\\
Pabell\'on I, Ciudad Universitaria, 1428 Buenos Aires, Argentina
\\  
{\tt vbecher@dc.uba.ar}
 \bigskip

\end{document}